\documentclass[12pt]{article}

\usepackage{amsmath,amsthm,amsfonts,amssymb,amscd}

\allowdisplaybreaks[4]

\newtheorem {Lemma}{Lemma}[section]
\newtheorem {Theorem} {Theorem}[section]

\newtheorem{Proposition}{Proposition}[section]

\numberwithin{equation}{section}

\allowdisplaybreaks [4]
\headsep =
0.0in \headheight = 0.0in \topmargin = 0.3in

\begin{document}

\title{Sharp bounds for ordinary and signless Laplacian spectral radii of uniform hypergraphs}

\author{Hongying Lin\footnote{E-mail: lhongying0908@126.com}, Biao Mo\footnote{E-mail: 172895568@qq.com}, Bo Zhou\footnote{Corresponding author. E-mail: zhoubo@scnu.edu.cn}, Weiming Weng\footnote{E-mail: jshwwm@163.com}\\
School of  Mathematical Sciences, South China Normal University,\\
Guangzhou 510631, P.R. China}

\date{}
\maketitle

\begin{abstract}

We give sharp upper bounds for the ordinary spectral radius and signless Laplacian spectral radius of a uniform hypergraph in terms of the average $2$-degrees or degrees of vertices, respectively, and we also give a lower  bound for the ordinary spectral radius. We also compare these bounds with known ones.
 \\ \\
 %{\bf AMS classifications:} 15A18, 05C65\\ \\
\end{abstract}

\section{Introduction}

For positive integers $k$ and $n$ with $k\le n$, a tensor $\mathcal{T}=(T_{i_1\dots i_k})$ of order $k$ and dimension $n$ refers to a multidimensional array
with complex entries $T_{i_1\dots i_k}$ for $i_j \in [n]:=\{1,\dots, n\}$ and  $j\in[k]$.
Obviously, a  vector is a  tensor of order $1$ and a matrix is a tensor of order $2$.
%Let $\mathcal{I}$ be the identity tensor of order $m$ and  appropriate dimension, e.g., $\mathcal{I}_{i_1\dots i_m}=1$ if %$i_1=\dots=i_m$ and $\mathcal{I}_{i_1\dots i_m}=0$ otherwise.

Let $\mathcal{M}$  be a  tensor of order $s\geq 2$  and  dimension $n$, and $\mathcal{N}$ a tensor of order $k\ge 1$ and  dimension $n$.
The product $\mathcal{M}\mathcal{N}$ is the tensor of order $(s-1)(k-1)+1$ and dimension $n$ with entries \cite{Sh}
\[
(\mathcal{M}\mathcal{N})_{ij_{1}\dots j_{s-1}}=\sum_{i_{2},\dots ,i_{s}\in [n]}M_{ii_{2}\dots i_{s}}N_{i_{2}j_{1}}\cdots N_{i_{s}j_{s-1}},
\]
with $i\in [n]$ and $j_{1},\dots ,j_{s-1}\in [n]^{k-1}$.

For a tensor $\mathcal{T}$ of order $k\ge 2$ and dimension $n$ and  a vector $x=(x_1,\dots, x_n)^\top$,  $\mathcal{T}x$ is an $n$-dimensional vector whose $i$-th entry is
\[(\mathcal{T}x)_i=\sum_{i_2,\dots ,i_k\in [n]} T_{ii_2\dots i_k}x_{i_2}\cdots x_{i_k},\]
where $i\in[n]$. Let $x^{[r]}=(x_1^r,\dots,x_n^r)^\top$.
For some complex $\rho$, if there is a nonzero $n$-dimensional vector $x$ such that
\[\mathcal{T}x=\rho x^{[k-1]},\]
then $\rho$ is called an eigenvalue of $\mathcal{T}$, and $x$ an eigenvector of $\mathcal{T}$ corresponding to $\rho$, see \cite{Q1,Qi}.
%In particular, if $x$ is real, then $\rho$ is also real. In this case, $\rho$ is called an eigenvalue of %$\mathcal{T}$, and $x$ an eigenvector of $\mathcal{T}$
%corresponding to $\rho$.
Let $\rho(\mathcal{T})$ be the largest modulus of the eigenvalues of $\mathcal{T}$.

Let $\mathcal{G}$ be a hypergraph with vertex set $V(\mathcal{G})=[n]$ and edge set $E(\mathcal{G})$, see~\cite{Be}.
%If $|e|=k$ for $e\in E(\mathcal{G})$,
If every edge of $\mathcal{G}$ has cardinality $k$,
then we say that $\mathcal{G}$ is a $k$-uniform hypergraph.
Throughout this paper, we consider $k$-uniform hypergraphs on $n$ vertices with $2\le k\le n$. A uniform hypergraph is a hypergraph that is $k$-uniform
for some $k$.
For $i\in [n]$,  $E_i$ denotes the set of edges of $\mathcal{G}$ containing $i$.
The degree of a vertex $i$ in $\mathcal{G}$ is defined as $d_i=|E_i|$. If $d_i=d$ for $i\in V(\mathcal{G})$, then $\mathcal{G}$ is called a regular hypergraph (of degree $d$).
For $i,j\in V(\mathcal{G})$, if there is a sequence of edges $e_1,\dots,e_r$ such that $i\in e_1$, $j\in e_r$ and $e_s\cap e_{s+1}\neq \emptyset $ for all $s\in[r-1]$, then we say that $i$ and $j$ are connected.  A hypergraph is  connected if every pair of different vertices of $\mathcal{G}$ is connected.

The adjacency tensor of a  $k$-uniform hypergraph $\mathcal{G}$ on $n$ vertices is defined as the  tensor $\mathcal{A}(\mathcal{G})$ of order $k$ and dimension $n$  whose $(i_1\dots i_k)$-entry is
\[A_{i_1 \dots i_k}= \begin{cases}
\frac{1}{(k-1)!} & \text{if } \{i_1, \dots, i_k\}\in E(\mathcal{G}),\\
0  & \text{otherwise}.
\end{cases}\]
Let $\mathcal{D}(\mathcal{G})$ be the diagonal tensor of order $k$ and dimension $n$ with its diagonal entry $D_{i\dots i}$ the degree of vertex $i$ for $i\in[n]$.
Then $\mathcal{Q}(\mathcal{G})=\mathcal{D}(\mathcal{G})+\mathcal{A}(\mathcal{G})$ is the signless Laplacian tensor of  $\mathcal{G}$.
We call $\rho(\mathcal{A}(\mathcal{G}))$ the (ordinary) spectral radius of $\mathcal{G}$, which is denoted by $\rho(\mathcal{G})$, and $\rho(\mathcal{Q}(\mathcal{G}))$ the signless Laplacian spectral radius of $\mathcal{G}$, which is denoted by $\mu (\mathcal{G})$.

For a nonnegative tensor $\mathcal{T}$ of order $k\ge 2$ and dimension $n$, the $i$-th row sum of $\mathcal{T}$ is $r_i(\mathcal{T})=\sum_{i_2,\dots,i_k\in [n]}T_{ii_2\dots i_k}$. If $r_i(\mathcal{T})>0$, then the $i$-th average $2$-row sum of $\mathcal{T}$ is defined as
\[m_i(\mathcal{T})=\frac{\sum_{i_2,\dots,i_k\in [n]}T_{ii_2\dots i_k}r_{i_2}(\mathcal{T})\cdots r_{i_k}(\mathcal{T})}{r_i^{k-1}(\mathcal{T})}.\]

Let $\mathcal{G}$ be a  $k$-uniform  hypergraph on $n$ vertices. Let $\mathcal{A}=\mathcal{A}(\mathcal{G})$. For $i\in V(\mathcal{G})$ with $d_i>0$,
\begin{eqnarray*}
m_i(\mathcal{A})&=&\frac{\sum_{i_2,\dots,i_k\in [n]}A_{ii_2\dots i_k}r_{i_2}(\mathcal{A})\cdots r_{i_k}(\mathcal{A})}{r_i^{k-1}(\mathcal{A})}\\
&=&\frac{\sum_{\{i,i_2,\dots,i_k\}\in E_i}d_{i_2}\cdots d_{i_k}}{d_i^{k-1}},
\end{eqnarray*}
which is called the average $2$-degree of vertex $i$ of $\mathcal{G}$ (average of degrees of vertices in $E_i$) \cite{YZL}.%, and
%\begin{eqnarray*}
%m_i(\mathcal{Q})&=&\frac{\sum_{i_2,\dots,i_k\in [n]}Q_{ii_2\dots i_k}r_{i_2}(\mathcal{Q})\cdots r_{i_k}(\mathcal{Q})}{r_i^{k-1}(\mathcal{Q})}\\
%&=&\frac{D_{i\ldots i}(2d_i)^{k-1}+\sum_{\{i,i_2,\dots,i_k\}\in E_i}(2d_{i_2})\cdots(2 d_{i_k})}{(2 d_i)^{k-1}}\\
%&=&d_i+\frac{\sum_{\{i,i_2,\dots,i_k\}\in E_i}d_{i_2}\cdots d_{i_k}}{d_i^{k-1}},
%\end{eqnarray*}
%which is called the signless Laplacian average $2$-degree of vertex $i$ of $\mathcal{G}$.

For a $k$-uniform hypergraph $\mathcal{G}$ with maximum degree $\Delta$, we know that $\rho(\mathcal{G})\le \Delta$ \cite{CD} and $\mu(\mathcal{G})\le 2\Delta$ \cite{Qi} with either equality when $\mathcal{G}$ is connected if and only if $\mathcal{G}$ is regular (see \cite{QSW}). Recently, upper bounds for $\rho(\mathcal{G})$ and $\mu(\mathcal{G})$ are given in \cite{YZL} using degree sequence. In this note, we present  sharp upper bounds for $\rho(\mathcal{G})$ and $\mu(\mathcal{G})$  using average $2$-degrees or degrees, and we also give a lower bound for $\rho(\mathcal{G})$. We compare these bounds with known bounds by examples.

\section{Preliminaries}

A nonnegative tensor $\mathcal{T}$ of order $k\ge 2$ dimension $n$ is called weakly irreducible if the associated  directed graph $D_\mathcal{T}$ of $\mathcal{T}$ is strongly connected, where $D_\mathcal{T}$ is the directed graph with vertex set $\{1, \dots, n\}$ and arc set $\{(i,j): a_{ii_2\dots i_k}\ne 0 \mbox{ for some } i_s=j \mbox{ with } s=2, \dots, k\}$ \cite{FGH,Qi}.

For an $n$-dimensional real vector $x$, let $\|x\|_k=(\sum_{i=1}^n|x_i|^k)^{\frac{1}{k}}$, and  if $\|x\|_k=1$, then we say that $x$ is a unit vector. Let $\mathbb{R}^n_{+}$ be the set of  $n$-dimensional nonnegative vectors.

\begin{Lemma}\label{irreducibility}\cite{FGH,YY}
%Let $\mathcal{T}$ be a weakly irreducible nonnegative tensor of order $k\ge 2$. Then there is
%a unique unit positive eigenvector corresponding to $\rho(\mathcal{T})$.
%
Let $\mathcal{T}$ be a  nonnegative tensor. Then $\rho(\mathcal{T})$ is an eigenvalue of $\mathcal{T}$ and there is a unit nonnegative eigenvector corresponding to $\rho(\mathcal{T})$.
If furthermore $\mathcal{T}$ is weakly irreducible, then there is
a unique unit positive eigenvector corresponding to $\rho(\mathcal{T})$.

\end{Lemma}

\begin{Lemma} \cite{Qi} \label{QLQ}
Let $\mathcal{G}$ be a $k$-uniform hypergraph with $n$ vertices. Then
$\rho(\mathcal{G})=\max\{x^\top(\mathcal{A}(\mathcal{G})x): x\in\mathbb{R}^n_{+},\|x\|_{k}=1\}$.
\end{Lemma}

\begin{Lemma} \label{connec} \cite{PT,Qi} Let $\mathcal{G}$ be a $k$-uniform hypergraph. Then $\mathcal{A}(\mathcal{G})$ ($\mathcal{Q}(G)$, respectively) is weakly irreducible if and only if $\mathcal{G}$ is connected.
\end{Lemma}

A hypergraph $\mathcal{H}$ is a subhypergraph of $\mathcal{G}$ if $V(\mathcal{H})\subseteq V(\mathcal{G})$ and  $E(\mathcal{H})\subseteq E(\mathcal{G})$.%, and $H$ is a proper subhypergraph of $\mathcal{G}$ if $V(H)\subsetneq V(\mathcal{G})$ and  $E(H)\subsetneq E(\mathcal{G})$.

\begin{Lemma} \cite{CD, FYZ} \label{subhypergraph}
Let $\mathcal{G}$ be a connected $k$-uniform hypergraph and $\mathcal{H}$ a subhypergraph of $\mathcal{G}$. Then
$\rho(\mathcal{H})\leq \rho(\mathcal{G})$ with equality if and only if $\mathcal{H}=\mathcal{G}$.
\end{Lemma}

For two tensors $\mathcal{M}$ and $\mathcal{N}$ of order $k\ge 2$ and dimension $n$, if  there is an $n\times n$ nonsingular diagonal matrix $U$ such that $\mathcal{N}=U^{-(k-1)}\mathcal{M}U$, then we say that $\mathcal{M}$ and $\mathcal{N}$ are  diagonal similar. %In this case, $\mathcal{B}_{i_{1}i_{2}\dots i_{k}}
%=U_{i_{1}i_{1}}^{-(k-1)}\mathcal{A}_{i_{1}i_{2}\dots i_{k}}U_{i_{2}i_{2}}\cdots U_{i_{k}i_{k}}$ and thus $D_{\mathcal{A}}=D_{\mathcal{B}}$, implying that $\mathcal{A}$ and $\mathcal{B}$ have the same weak irreducibility.

\begin{Lemma}\cite{Sh} \label{Tensor1}
Let $\mathcal{M}$ and $\mathcal{N}$ be two   diagonal similar tensors of order $k\ge 2$ and dimension $n$. Then $\mathcal{M}$  and $\mathcal{N}$
have the same  real eigenvalues.
\end{Lemma}

\begin{Lemma}\label{Tensor2} \cite{LCL,YY}
Let $\mathcal{T}$ be a nonnegative tensor of order $k\ge 2$ and dimension $n$. Then
\[\min_{1\leq i \leq n}r_i(\mathcal{T})\leq \rho(\mathcal{T})\leq \max_{1\leq i \leq n}r_i(\mathcal{T}).\]
Moreover, if $\mathcal{T}$ is weakly irreducible, then either equality holds if and only if $r_1(\mathcal{T})=\cdots=r_n(\mathcal{T})$.
\end{Lemma}

\begin{Proposition}\label{Tensor3}
Let $\mathcal{T}$ be a nonnegative tensor of order $k\ge 2$ and dimension $n$ with all row sums positive. Then
\[\min_{1\leq i \leq n}m_i(\mathcal{T})\leq \rho(\mathcal{T})\leq \max_{1\leq i \leq n}m_i(\mathcal{T}).\]
Moreover, if $\mathcal{T}$ is weakly irreducible, then either equality holds if and only if $m_1(\mathcal{T})=\cdots=m_n(\mathcal{T})$.
\end{Proposition}

\begin{proof} Let $U=diag (r_1(\mathcal{T}),\dots, r_n(\mathcal{T}))$ and $\mathcal{B}=U^{-(k-1)}\mathcal{T}U$.
Then $\mathcal{T}$ and $\mathcal{B}$ are diagonal similar, and thus we have by Lemma~\ref{Tensor1} that $\rho(\mathcal{T})=\rho(\mathcal{B})$.
Obviously,
\[B_{i_1\dots i_k}=\frac{T_{i_1\dots i_k}r_{i_2}(\mathcal{T})\cdots r_{i_k}(\mathcal{T})}{r_{i_1}^{k-1}(\mathcal{T})}\]
for $i_1,i_2,\dots, i_k\in[n]$.
Thus
\[r_i(\mathcal{B})=\frac{\sum_{i_2,\dots, i_k\in[n]}T_{ii_2\dots i_k}r_{i_2}(\mathcal{T})\cdots r_{i_k}(\mathcal{T})}{r_{i}^{k-1}(\mathcal{T})}=m_i(\mathcal{T})\]
for $i\in[n]$.
By Lemma~\ref{Tensor2}, we have \[\min_{1\leq i \leq n}m_i(\mathcal{T})=\min_{1\leq i \leq n}r_i(\mathcal{B})\leq \rho(\mathcal{T})=\rho(\mathcal{B})\leq \max_{1\leq i \leq n}r_i(\mathcal{B})=\max_{1\leq i \leq n}m_i(\mathcal{T}),\]
and if $\mathcal{T}$ is weakly irreducible, then since $D_\mathcal{T}=D_\mathcal{B}$, $\mathcal{B}$ is also weakly irreducible, and thus $\rho(\mathcal{T})=\min_{1\leq i \leq n}m_i(\mathcal{T})$ or $\rho(\mathcal{T})=\max_{1\leq i \leq n}m_i(\mathcal{T})$ if and only if $r_1(\mathcal{B})=\cdots=r_n(\mathcal{B})$, i.e., $m_1(\mathcal{T})=\cdots=m_n(\mathcal{T})$.
\end{proof}

For a  hypergraph $\mathcal{G}$, the blow-up of $\mathcal{G}$, denoted by $\mathcal{G}^1$, is the hypergraph obtained from $\mathcal{G}$ by adding a new common vertex $v$ to each edge. If $\mathcal{G}$ is a regular $(k-1)$-uniform hypergraph on $n-1$ vertices of degree $d$, then  $\mathcal{G}^1$ is a $k$-uniform hypergraph on $n$ vertices.

We use the techniques in \cite{YZL}.

\section{Main results}

Let $\mathcal{G}$ be a $k$-uniform hypergraph on $n$ vertices  without isolated vertices with average $2$-degrees $m_1\geq \dots \geq m_n$. By Proposition \ref{Tensor3},
$\rho(\mathcal{G})\leq m_1$. In the following, we give a upper bound for $\rho(\mathcal{G})$ using $m_1$ and $m_2$.

\begin{Theorem}\label{Tensor5}
Let $\mathcal{G}$ be a $k$-uniform hypergraph on $n$ vertices without isolated vertices with average $2$-degrees $m_1\geq  \dots \geq m_n$.
Then
\begin{equation} \label{HY}
\rho(\mathcal{G})\leq m_1^{\frac{1}{k}}m_{2}^{1-\frac{1}{k}}.
\end{equation}
%Then \[m_n^{\frac{1}{k}}m_{n-1}^{1-\frac{1}{k}}\leq\rho(\mathcal{G})\leq m_1^{\frac{1}{k}}m_{2}^{1-\frac{1}{k}}.\]
Moreover, if $\mathcal{G}$ is connected,  then equality holds in (\ref{HY}) if and only  if each vertex of $\mathcal{G}$ has the same average $2$-degree.% or the blow-up hypergraph of some $(k-1)$-uniform hypergraph on $n-1$ vertices of the same average $2$-degree.
\end{Theorem}

\begin{proof} Let $\mathcal{A}=\mathcal{A}(\mathcal{G})$.

If $m_1=m_2$, then by Proposition~\ref{Tensor3}, we have \[\rho(\mathcal{G})\leq m_1^{\frac{1}{k}}m_{2}^{1-\frac{1}{k}}=m_1,\]
and when $\mathcal{G}$ is connected, $A$ is weakly irreducible, and thus  %, and then equality holds if and only if $m_1=\cdots=m_n$
equality holds in (\ref{HY}) if and only  if each vertex of  $\mathcal{G}$ has the same average $2$-degree.

Suppose in the following that $m_1>m_2$.
Let $d_1, d_2,\dots, d_n$ be the degree sequence of $\mathcal{G}$.
Let $U$ be diagonal matrix $\mbox{diag}(td_1, d_2, \dots, d_n)$, where $t>1$ is a variable to be determined later.
Let  $\mathcal{T}=U^{-(k-1)}\mathcal{A}U$.
Then $\mathcal{A}$ and $\mathcal{T}$ are diagonal similar.
By Lemma~\ref{Tensor1}, $\mathcal{A}$ and $\mathcal{T}$ have the same real eigenvalues. By Lemma~\ref{irreducibility}, $\rho(\mathcal{A})$ is an eigenvalue of $\mathcal{A}$  and $\rho(\mathcal{T})$ is an eigenvalue of $\mathcal{T}$. Thus $\rho(\mathcal{G})=\rho(\mathcal{A})=\rho(\mathcal{T})$.
Obviously,
\[T_{i_1\dots i_k}=U^{-(k-1)}_{i_1i_1}A_{i_1\dots i_k}U_{i_2i_2}\cdots U_{i_ki_k}\]
for $i_1,\dots, i_k\in[n]$.
Then
\begin{eqnarray*}
r_1(\mathcal{T})&=& \sum_{i_2,\dots, i_k\in[n]}T_{ii_2\dots i_k}\\
&=&\sum_{i_2,\dots, i_k\in[n]}U^{-(k-1)}_{11}A_{1i_2\dots i_k}U_{i_2i_2}\cdots U_{i_ki_k}\\
&=& \sum_{i_2,\dots, i_k\in[n]\setminus\{1\}}(td_1)^{-(k-1)}A_{1i_2\dots i_k}d_{i_2}\cdots d_{i_k}\\
&=&\frac{\sum_{\{1,i_2,\dots, i_k\}\in E_1}d_{i_2}\cdots d_{i_k}}{(t d_1)^{k-1}}\\
&=& \frac{m_1}{t^{k-1}}.
\end{eqnarray*}
For $i=2, \dots, n$, let \[m_{1,i}=m_i-\frac{\sum_{1\notin\{i,i_2,\dots,i_k\}\in E_i}d_{i_2}\cdots d_{i_k}}{d_i^{k-1}},\]
and then
\begin{eqnarray*}
r_i(\mathcal{T})&=& \sum_{i_2,\dots, i_k\in[n]}T_{ii_2\dots i_k}\\
&=&\sum_{i_2,\dots, i_k\in[n]}U^{-(k-1)}_{ii}A_{ii_2\dots i_k}U_{i_2i_2}\cdots U_{i_ki_k}\\
&=&\sum_{i_2,\dots, i_k\in[n]\atop 1\in\{i_2,\dots, i_k\}}d_i^{-(k-1)}A_{ii_2\dots i_k}d_{i_2}\cdots d_{i_k}+\sum_{i_2,\dots, i_k\in[n]\atop 1\notin\{i_2,\dots, i_k\}}d_i^{-(k-1)}A_{ii_2\dots i_k}d_{i_2}\cdots d_{i_k}\\
&=&\frac{\sum_{i_2,\dots, i_k\in[n]\atop 1\in\{i_2,\dots, i_k\}}A_{ii_2\dots i_k}d_{i_2}d_{i_3}\cdots d_{i_k}}{d_i^{k-1}}+\frac{\sum_{i_2,\dots, i_k\in[n]\atop 1\notin\{i_2,\dots, i_k\}}A_{ii_2\dots i_k}d_{i_2}\cdots d_{i_k}}{d_i^{k-1}}\\
&=&\frac{\sum_{\{i,1,i_3,\dots, i_k\}\in E_i}(td_1)d_{i_3}\cdots d_{i_k}}{d_i^{k-1}}+\frac{\sum_{1\notin\{i,i_2,\dots,i_k\}\in E_i}d_{i_2}\cdots d_{i_k}}{d_i^{k-1}}\\
&=&\frac{t\sum_{\{i,1,i_3,\dots, i_k\}\in E_i} d_1d_{i_3}\cdots d_{i_k}}{d_i^{k-1}}+\frac{\sum_{1\notin\{i,i_2,\dots,i_k\}\in E_i}d_{i_2}\cdots d_{i_k}}{d_i^{k-1}}\\
&=& tm_{1,i}+m_i-m_{1,i}\\
&\leq& tm_i\\&\leq& tm_2,
\end{eqnarray*}
with equality if and only if $m_i=m_{1,i}$ and $m_i=m_2$.
Take $t=\left(\frac{m_1}{m_2}\right)^\frac{1}{k}$. Obviously, $t> 1$.
Then $r_1(\mathcal{T})=m_1^{\frac{1}{k}}m_{2}^{1-\frac{1}{k}}$ and
for $2\leq i\leq n$, $r_i(\mathcal{T})\leq tm_2=m_1^{\frac{1}{k}}m_{2}^{1-\frac{1}{k}}$. By Lemma~\ref{Tensor2},
 \[\rho(\mathcal{G})=\rho(\mathcal{T})\leq m_1^{\frac{1}{k}}m_{2}^{1-\frac{1}{k}}.\]

Now suppose that $\mathcal{G}$ is connected. By Lemma \ref{connec}, $\mathcal{A}(\mathcal{G})$ is weakly irreducible, and thus $\mathcal{T}$ is weakly irreducible since $D_{\mathcal{A}(\mathcal{G})}=D_{\mathcal{T}}$.

Suppose that  equality holds in (\ref{HY}). By Lemma \ref{Tensor2}, $r_1(\mathcal{T})=\cdots=r_n(\mathcal{T})$.
From the above argument, we have (i) $m_2=\cdots=m_n$, and (ii) $m_i=m_{1,i}$ for $i=2,\ldots,n$.  From (ii) and the definition of $m_{1,i}$, each edge of $\mathcal{G}$ contains vertex $1$, which implies that $d_1(k-1)=\sum_{i=2}^nd_i$. From (i),
\[m_1=\frac{\sum_{i=2}^n \frac{d^{k}_im_i}{d_1}}{d_1^{k-1}(k-1)}=\frac{\sum_{i=2}^n d^k_im_2}{d_1^k}=\frac{\sum_{i=2}^n d_im_2\left(\frac{d_i}{d_k}\right)^{k-1}}{d_1(k-1)}\leq\frac{\sum_{i=2}^nm_2d_i}{d_1(k-1)} =m_2,\]
a contradiction. Thus the inequality (\ref{HY})  is strict  if $m_1>m_2$.
\end{proof}

Let $\mathcal{G}$ be a connected $k$-uniform hypergraph on $n$ vertices with degree sequence $d_1\ge \cdots \ge d_n$. Then \cite{YZL}
\begin{equation} \label{old}
\rho(\mathcal{G})\le d_1^{\frac{1}{k}}d_2^{1-\frac{1}{k}}
\end{equation}
with equality if and only if $\mathcal{G}$ is a regular hypergraph or the blow-up hypergraph $\mathcal{H}^1$ of a regular $(k-1)$-uniform hypergraph $\mathcal{H}$ on $n-1$ vertices.

%If $m_1>m_2$, then the upper bound for $\rho(\mathcal{G})$  in Theorem~\ref{Tensor5} is better than that in Corollary~\ref{Tensor4}.

Obviously, if $\mathcal{G}$ is a regular hypergraph, then each vertex of $\mathcal{G}$ has the same average $2$-degree, and thus equality holds in (\ref{HY}) and (\ref{old}). For the blow-up hypergraph $\mathcal{H}^1$ of a regular $(k-1)$-uniform hypergraph $\mathcal{H}$ on $n-1$ vertices, the upper bound in (\ref{old}) is attained, while the upper bound in (\ref{HY}) is not attained. However,
in the following, we give two examples to show that
there are irregular hypergraphs for which each vertex has the same average $2$-degree. For such hypergraphs, the upper in (\ref{HY}) is attained, while the upper bound in (\ref{old}) can not be attained.

%\noindent {\bf  Example~1.}
Let $\mathcal{H}_1$ be a $3$-uniform hypergraph with vertex set $V(\mathcal{H}_1)=[34]$ and $E(\mathcal{H}_1)=\{e_i:1\leq i \leq 51\}$, where
\begin{eqnarray*}
\begin{matrix}
e_1=\{1,2,5\}, & e_2=\{1,2,6\}, & e_3=\{1,2,7\}, &
e_4=\{1,2,8\},\\
e_5=\{1,2,9\}, & e_6=\{1,2,10\}, & e_7=\{1,2,11\}, & e_8=\{1,2,12\},\\
e_9=\{1,2,13\}, & e_{10}=\{3,4,5\}, & e_{11}=\{3,4,6\}, & e_{12}=\{3,4,7\},\\
e_{13}=\{3,4,8\}, & e_{14}=\{3,4,9\},& e_{15}=\{3,4,10\},& e_{16}=\{3,4,11\},\\
e_{17}=\{3,4,12\}, & e_{18}=\{3,4,13\}, & e_{19}=\{5,6,14\}, & e_{20}=\{6,7,15\}, \\
e_{21}=\{7,8,16\}, & e_{22}=\{8,9,17\}, & e_{23}=\{9,10,18\}, & e_{24}=\{10,11,19\}, \\
e_{25}=\{11,12,20\}, & e_{26}=\{12,13,21\}, & e_{27}=\{13,5,22\}, & e_{28}=\{5,23,24\}, \\
e_{29}=\{5,25,26\}, & e_{30}=\{6,27,28\}, & e_{31}=\{6,29,30\}, & e_{32}=\{7,31,32\}, \\
e_{33}=\{7,33,34\}, & e_{34}=\{8,23,24\}, & e_{35}=\{8,25,26\}, & e_{36}=\{9,27,28\}, \\
e_{37}=\{9,29,30\}, & e_{38}=\{10,31,32\}, & e_{39}=\{10,33,34\}, & e_{40}=\{11,23,24\}, \\
e_{41}=\{11,25,26\}, & e_{42}=\{12,27,28\}, & e_{43}=\{12,29,30\}, & e_{44}=\{13,31,32\}, \\
e_{45}=\{13,33,34\}, & e_{46}=\{14,15,16\}, & e_{47}=\{17,18,19\}, & e_{48}=\{20,21,22\}, \\
e_{49}=\{14,17,20\}, & e_{50}=\{15,18,21\}, & e_{51}=\{16,19,22\}.
\end{matrix}
\end{eqnarray*}
By direct calculation, we have
%$d_i=9$ for $1\leq i\leq 4$, $d_i=6$ for $5\leq i\leq 13$, $d_i=3$ for $14\leq i\leq 34$,
\[d_i= \begin{cases}
9& \text{if } 1\leq i\leq 4,\\
6 & \text{if } 5\leq i\leq 13,\\
3 & \text{if } 14\leq i\leq 34,
\end{cases}\]
and
\begin{eqnarray*}
m_i&=& \begin{cases}
\frac{(9\times 6)\times 9}{9\times9}& \text{if } 1\leq i\leq 4,\\
\frac{(9\times 9)\times 2+(6\times 3)\times 2+(3\times 3)\times 2}{6\times6} & \text{if } 5\leq i\leq 13,\\
\frac{6\times 6+(3\times 3)\times 2}{3\times3} & \text{if } 14\leq i\leq 22,\\
\frac{(6\times 3)\times 3}{3\times3} & \text{if } 23\leq i\leq 34
\end{cases}\\
&=&6.
\end{eqnarray*}
By Theorem~\ref{Tensor5}, we have $\rho(\mathcal{H}_1)=6$.
%$d_i=9$ for $1\leq i\leq 4$, $d_i=6$ for $5\leq i\leq 13$, $d_i=3$ for $14\leq i\leq 34$, and $m_i(\mathcal{H}_1)=\frac{(9\times 6)\times 9}{9\times9}=6$ for $1\leq i\leq 4$,
%$m_i(\mathcal{H}_1)=\frac{(9\times 9)\times 2+(6\times 3)\times 2+(3\times 3)\times 2}{6\times6}=6$ for $5\leq i\leq 13$,
%$m_i(\mathcal{H}_1)=\frac{6\times 6+(3\times 3)\times 2}{3\times3}=6$ for $14\leq i\leq 22$,
%$m_i(\mathcal{H}_1)=\frac{(6\times 3)\times 3}{3\times3}=6$ for $23\leq i\leq 34$.

Let $\mathcal{H}_2$ be a $3$-uniform hypergraph with vertex set $V(\mathcal{H}_2)=[54]$ and $E(\mathcal{H}_2)=\{e_i:1\leq i \leq 64\}$, where
\begin{eqnarray*}
\begin{matrix}
e_1=\{1,2,3\}, & e_2=\{3,4,6\}, & e_3=\{3,4,5\}, & e_4=\{4,5,6\}, \\
e_5=\{1,2,6\}, & e_6=\{1,2,5\}, & e_7=\{1,3,4\}, & e_8=\{2,5,6\},\\
e_9=\{1,7,8\}, & e_{10}=\{1,8,9\}, & e_{11}=\{1,9,10\}, & e_{12}=\{1,10,11\},\\
e_{13}=\{2,11,12\}, & e_{14}=\{2,12,13\}, & e_{15}=\{2,13,14\}, & e_{16}=\{2,14,15\},\\
e_{17}=\{3,15,16\}, & e_{18}=\{3,16,17\}, & e_{19}=\{3,17,18\}, & e_{20}=\{3,18,19\}, \\
e_{21}=\{4,19,20\}, & e_{22}=\{4,20,21\}, & e_{23}=\{4,21,22\}, & e_{24}=\{4,22,23\}, \\
e_{25}=\{5,23,24\}, & e_{26}=\{5,24,25\}, & e_{27}=\{5,25,26\}, & e_{28}=\{5,26,27\}, \\
e_{29}=\{6,27,28\}, & e_{30}=\{6,28,29\}, & e_{31}=\{6,29,30\}, & e_{32}=\{6,30,7\}, \\
e_{33}=\{7,8,31\}, & e_{34}=\{7,8,32\}, & e_{35}=\{9,10,33\}, & e_{36}=\{9,10,34\}, \\
e_{37}=\{11,12,35\}, & e_{38}=\{11,12,36\}, & e_{39}=\{13,14,37\}, & e_{40}=\{13,14,38\},\\
e_{41}=\{15,16,39\}, & e_{42}=\{15,16,40\}, & e_{43}=\{17,18,41\}, & e_{44}=\{17,18,42\},\\
e_{45}=\{19,20,43\}, & e_{46}=\{19,20,44\}, & e_{47}=\{21,22,45\}, & e_{48}=\{21,22,46\},\\
e_{49}=\{23,24,47\}, & e_{50}=\{23,24,48\}, & e_{51}=\{25,26,49\}, & e_{52}=\{25,26,50\},\\
e_{53}=\{27,28,51\}, & e_{54}=\{27,28,52\}, & e_{55}=\{29,30,53\}, & e_{56}=\{29,30,54\}, \\
e_{57}=\{31,32,33\}, & e_{58}=\{34,35,36\}, & e_{59}=\{37,38,39\}, & e_{60}=\{40,41,42\}, \\
e_{61}=\{43,44,45\}, & e_{62}=\{46,47,48\}, & e_{63}=\{49,50,51\}, & e_{64}=\{52,53,54\}.
\end{matrix}
\end{eqnarray*}
By direct calculation, we have %$d_i=8$ for $1\leq i\leq 6$, $d_i=4$ for $7\leq i\leq 30$, $d_i=2$ for $31\leq i\leq 54$
\[d_i= \begin{cases}
8& \text{if } 1\leq i\leq 6,\\
4 & \text{if } 7\leq i\leq 30,\\
2 & \text{if } 31\leq i\leq 54,
\end{cases}\]
and
\begin{eqnarray*}
m_i&=& \begin{cases}
\frac{(8\times 8)\times 4+(4\times 4)\times 4}{8\times8}& \text{if } 1\leq i\leq 6,\\
\frac{(8\times 4)\times 2+(4\times 3)\times 2}{4\times4} & \text{if } 7\leq i\leq 30,\\
\frac{4\times 4+2\times 2}{2\times2} & \text{if } 31\leq i\leq 54
\end{cases}\\
&=&5.
\end{eqnarray*}
By Theorem~\ref{Tensor5}, we have $\rho(\mathcal{H}_2)=5$.

%$d_i=8$ for $1\leq i\leq 6$, $d_i=4$ for $7\leq i\leq 30$, $d_i=2$ for $31\leq i\leq 54$, and $m_i(\mathcal{H}_2)=5$ for $1\leq i\leq 54$.

%\noindent {\bf  Example~3.}

%\begin{Corollary}\label{Tensor6}
Let $\mathcal{G}$ be a $k$-uniform hypergraph of order $n$  without isolated vertices with maximum degree $\Delta$ and  average $2$-degrees $m_1\geq \dots \geq m_n$. Note that $\mu(\mathcal{G})\le \Delta+\rho(\mathcal{G})$. By Theorem~\ref{Tensor5}, we have
Then $\mu(\mathcal{G})\leq m_1^{\frac{1}{k}}m_{2}^{1-\frac{1}{k}}+\Delta$.
%\end{Corollary}

If we  take $U=\mbox{diag} (d_1, \ldots, d_{n-1},yd_n)$
with $y=\left(\frac{m_n}{m_{n-1}}\right)^\frac{1}{k}$ in the proof of Theorem~\ref{Tensor5}, then $\rho(\mathcal{G})\geq m_n^{\frac{1}{k}}m_{n-1}^{1-\frac{1}{k}}$,
and if $\mathcal{G}$ is connected,  then equality holds if and only  if each vertex of $\mathcal{G}$ has the same average $2$-degree.

For a $k$-uniform hypergraph $\mathcal{G}$, if there is a disjoint partition of $V(\mathcal{G})$ as $V(\mathcal{G})=V_0\cup V_1\cup\dots\cup V_d$, where $|V_0|=1,|V_1|=\dots=|V_d|=k-1$, and $E(\mathcal{G})=\{V_0\cup V_i:i\in[d]\}$, then $\mathcal{G}$ is called a hyperstar, denoted by $\mathcal{S}_d^k$. The  vertex (of degree $d$) in $V_0$  is called the heart. Obviously, it is an isolated vertex if $d=0$.

For positive integers $d_1,\gamma$ and nonnegative integer $d_2$,  let $\mathcal{G}_{d_1,d_2,\gamma}$ be the $k$-uniform hypergraph obtained vertex-disjoint $\mathcal{S}_{d_1}^k$ and $\mathcal{S}_{d_2}^k$ by adding $\gamma(k-2)$ new vertices $v_{1,1},\dots ,v_{1,k-2},\dots,$ $v_{\gamma,1},\dots ,v_{\gamma,k-2} $ and $\gamma$ new edges $e_1,\dots,e_{\gamma}$,
 where $e_i=\{u,v,v_{i,1},\dots ,v_{i,k-2}\}$ for $ i\in[\gamma]$, and $u,v$ are the hearts of  $\mathcal{S}_{d_1}^k$ and $\mathcal{S}_{d_2}^k$, respectively. Obviously, if $d_2=0$ and $\gamma=1$, then $\mathcal{G}_{d_1,d_2,\gamma}\cong \mathcal{S}_{d_1+1}^k$.

Next we give a lower bound for $\rho(\mathcal{G})$ of a $k$-uniform hypergraph $\mathcal{G}$.

\begin{Theorem} \label{Tensor9}  Let $\mathcal{G}$ be a $k$-uniform  hypergraph  with $u\in V(\mathcal{G})$ of maximum degree $\Delta\ge 1$. Let $v$ be a  neighbor of $u$ with maximum degree. Then
\begin{equation} \label{star}
\rho(\mathcal{G})\geq \left(\frac{\Delta+\delta-2\gamma+\gamma^2+\sqrt{(\Delta-\delta)^2+\gamma^4+2(\Delta +\delta-2\gamma)\gamma^2}}{2}\right)^{\frac{1}{k}},
\end{equation}
where $\delta$ is the degree of $v$, and  $\gamma$ is the number of edges containing $u$ and $v$.
Moreover, if $\mathcal{G}$ is connected, then equality holds in (\ref{star})  if and only if $\mathcal{G}\cong \mathcal{G}_{\Delta,\delta,\gamma}$.
\end{Theorem}

\begin{proof}
Let $e_1, \dots, e_{\Delta}$ be the $\Delta$ edges of $\mathcal{G}$ containing  $u$. Among these edges, $\gamma$ of them, say  $e_1, \dots, e_{\gamma}$, contain $v$.
Let $e_{\Delta+1}, \dots, e_{\Delta+\delta-\gamma}$ be the $\delta-\gamma$ edges of $\mathcal{G}$ containing  $v$  different from $e_1, \dots, e_{\gamma}$.
Let $\mathcal{G}_1$ be the  subhypergraph of $\mathcal{G}$ induced by $\{e_1, \dots, e_{\Delta+\delta-\gamma}\}$. Then $V(\mathcal{G}_1)=\cup_{i=1}^{\Delta+\delta-\gamma}e_{i}$.
For $1\leq i\leq \Delta+\delta-\gamma$, let
$e_{i}=\{v_{i,1},\dots,v_{i,k}\}$, where $v_{i,1}=u$ and $v_{i,2}=v$ if  $1\leq i\leq \gamma$, $v_{i,1}=u$ if $\gamma+1\leq i\leq \Delta$, and $v_{i,1}=v$ if $\Delta+1\leq i\leq \Delta+\delta-\gamma$.
Note that maybe some of $v_{i,s}$ and $v_{j,t}$ for $1\leq s, t\leq k$ and  $1\leq i<j\leq \Delta+\delta-\gamma$  with $v_{i,s}, v_{j,t}\ne u,v$
represent the same vertex.

Let $\mathcal{G}'_1$ be a new hypergraph such that $V(\mathcal{G}'_1)=\cup_{i=1}^{\Delta+\delta-\gamma}e'_{i}$ and
$E(\mathcal{G}'_1)=\{e'_1, \dots, e'_{\Delta+\delta-\gamma}\}$, where
 $e'_{i}=\{v'_{i,1},\dots,v'_{i,k}\}$ with  $v'_{i,1}=u$ and $v'_{i,2}=v$ if $i=1, \dots, \gamma$,
$v'_{i,1}=u$ if $\gamma+1\leq i\leq \Delta$, and $v'_{i,1}=v$ if $\Delta+1\leq i\leq \Delta+\delta-\gamma$.
Note that $v\notin e_i$ for  $\gamma+1\leq i\leq \Delta$, $u\notin e_i$ for $\Delta+1\leq i\leq \Delta+\delta-\gamma$, and
$v'_{i,s}$  and $v'_{j,t}$ for $1\leq s, t\leq k$ and  $1\le i<j\le \Delta+\delta-\gamma$ with $v'_{i,s}, v'_{j,t}\ne u,v$ are different vertices.
Obviously, $\mathcal{G}'_1\cong \mathcal{G}_{\Delta,\delta,\gamma}$

By Lemma~\ref{irreducibility}, there is a unit positive eigenvector $x$ of $\mathcal{A}(\mathcal{G}'_1)$ corresponding to $\rho(\mathcal{G}'_1)$, in which the entry at $v'_{i,s}$ is denoted by $x_{i,s}$, where $1\leq i\leq \Delta+\delta-\gamma$ and $1\leq s\leq k$.  Then  $\rho(\mathcal{G}'_1)=x^\top (\mathcal{A}(\mathcal{G}'_1)x)$.
Let $w$ be any vertex of $\cup_{i=\Delta-\gamma+1}^{\Delta}e_{i}\setminus\{u\}$.
Since $\rho(\mathcal{G}'_1)x_w^{k-1}=x_u x^{k-2}_w$, we have $x_w=\frac{ x_u}{\rho(\mathcal{G}'_1)}$.
Thus the entry of $x$ at each vertex of $\cup_{i=\Delta-\gamma+1}^{\Delta}e_{i}\setminus\{u\}$ is the same, denoted by $a$.
Similarly, the entry of $x'$ at each vertex of $\cup_{i=1}^{\gamma}e_{i}\setminus\{u,v\}$ is the same, denoted by $b$,
and the entry of $x'$ at each vertex of
$\cup_{i=\Delta+1}^{\Delta+\delta-\gamma}e_{i}\setminus\{v\}$ is the same, denoted by $c$.
Then
\begin{eqnarray*}%\label{TT4}
\rho(\mathcal{G}'_1)a^{k-1}&=&  x_u a^{k-2},\\
\rho(\mathcal{G}'_1)x^{k-1}_u&=&  (\Delta-\gamma) a^{k-1}+\gamma b^{k-2}x_v,\\
\rho(\mathcal{G}'_1)b^{k-1}&=&  x_ux_vb^{k-3},\\
\rho(\mathcal{G}'_1)x^{k-1}_v&=& (\delta-\gamma) c^{k-1}+\gamma b^{k-2}x_u, \\
\rho(\mathcal{G}'_1)c^{k-1}&=& x_vc^{k-2}.
\end{eqnarray*}
Thus $\rho(\mathcal{G}'_1)$ is the largest root of the equation $f(\rho)=0$, where $f(\rho)=(\rho^k-\Delta+\gamma)\left(\rho^{2k}-(\Delta+\delta-2\gamma+\gamma^2)\rho^k+(\Delta-\gamma)(\delta-\gamma)\right)$.
It follows that \[\rho(\mathcal{G}'_1)=\left(\frac{\Delta+\delta-2\gamma+\gamma^2+\sqrt{(\Delta-\delta)^2+\gamma^4+2(\Delta +\delta-2\gamma)\gamma^2}}{2}\right)^{\frac{1}{k}}.\]

Construct a surjection $\sigma$ from $V(\mathcal{G}'_1)$ to $V(\mathcal{G}_1)$ such that $\sigma(v'_{i,s})=v_{i,s}$ for $1\leq i\leq \Delta +\delta-\gamma$ and  $1\leq s\leq k$.  Let $y=(y_1,\dots,y_{|V(\mathcal{G}_1)|})^{\top}$ such that $y_{i}=\max_{v'_{j,s}\in\sigma^{-1}(i)}\{x_{j,s}\}$ for $1\leq i\leq |V(\mathcal{G}_1)|$.
Obviously, $\|y\|_{k}\leq \|x\|_{k}=1$. Let $z=\frac{y}{\|y\|_{k}}$. Then $\|z\|_{k}=1$.
By Lemma~\ref{QLQ},
\begin{equation} \label{med}
\rho(\mathcal{G}_1)\geq z^\top (\mathcal{A}(\mathcal{G}_1)z) =\frac{y^\top (\mathcal{A}(\mathcal{G}_1)y)}{\|y\|_{k}^k}\geq\frac{x^\top (\mathcal{A}(\mathcal{G}'_1)x)}{\|x\|_{k}^k}= x^\top (\mathcal{A}(\mathcal{G}'_1)x)
=\rho(\mathcal{G}'_1).
\end{equation}
Since $\mathcal{G}_1$ is a subhypergraph of $\mathcal{G}$, we have by Lemma~\ref{subhypergraph} that $\rho(\mathcal{G})\geq\rho(\mathcal{G}_1)$.
Thus
 \[\rho(\mathcal{G})\geq\rho(\mathcal{G}'_1)= \left(\frac{\Delta+\delta-2\gamma+\gamma^2+\sqrt{(\Delta-\delta)^2+\gamma^4+2(\Delta +\delta-2\gamma)\gamma^2}}{2}\right)^{\frac{1}{k}}.\]

If $\mathcal{G}\cong \mathcal{G}_{\Delta,\delta,\gamma}$, then by the above proof, equality holds in (\ref{star}).

Suppose that $\mathcal{G}$ is connected and  equality holds in (\ref{star}). Then all equalities hold in (\ref{med}) and $\rho(\mathcal{G})=\rho(\mathcal{G}_1)$.
Thus by the construction of $\mathcal{G}'_1$, we have $\mathcal{G}_1 \cong \mathcal{G}'_1$. Otherwise, $|V(\mathcal{G}_1)| <|V(\mathcal{G}'_1)|$, and then $\|y\|_{k}< \|x\|_{k}=1$, a contradiction.
By Lemma~\ref{subhypergraph}, we have $\mathcal{G}=\mathcal{G}_1$. Thus $\mathcal{G}\cong \mathcal{G}_{\Delta,\delta,\gamma}$.
\end{proof}

Let $\mathcal{G}$ be a $k$-uniform  hypergraph  with maximum degree $\Delta\ge 1$. Let $f(\Delta, \delta, \gamma)$ be the  lower bound in (\ref{star}). For $\gamma\le \delta\le \Delta$, $f(\Delta, \delta, \gamma)$ is a increasing  function at $\delta$. Note that $\gamma\ge 1$.
By  Theorem~\ref{Tensor9},
$\rho(\mathcal{G})\geq f(\Delta, 1,1)=\Delta^{\frac{1}{k}}$,
and  if $\mathcal{G}$ is connected, then equality holds if and only if $\mathcal{G}$ is a hyperstar. Moreover, if $\mathcal{G}$ is connected and is not a hyperstar, then $\rho(\mathcal{G})\geq f(\Delta, 2,1)= \left(\frac{\Delta+1+\sqrt{\Delta^2-2\Delta+5}}{2}\right)^{\frac{1}{k}}$ with equality if and only if $\mathcal{G}\cong \mathcal{G}_{\Delta,2,1}$.

In the following, we give upper bounds for $\mu(\mathcal{G})$ of a $k$-uniform hypergraph.

\begin{Theorem} \label{Mo} Let $\mathcal{G}$ be a $k$-uniform  hypergraph on $n$ vertices with degree sequence $d_{1}\geq \dots \geq d_{n}$.  Let $d^*=1$ if $d_1=d_2$ and $d^*$ be  a root of $h(t)=0$ in $((\frac{d_{1}}{d_{2}})^{\frac{1}{k}},\frac{d_{1}}{d_{2}})$ if $d_1>d_2$, where $h(t)=d_{2}t^{k}+(d_{2}-d_{1})t^{k-1}-d_{1}$. Then
\begin{equation} \label{MZh}
\mu(\mathcal{G})\leq d_{1}+d_{1}\left(\frac{1}{d^*}\right)^{k-1}.
\end{equation}
 Moreover, if $\mathcal{G}$ is connected, then equality holds in (\ref{MZh}) if and only if $\mathcal{G}$ is a regular hypergraph or the blow-up hypergraph of a regular $(k-1)$-uniform hypergraph on $n-1$ vertices.
\end{Theorem}

\begin{proof} Let $\mathcal{Q}=\mathcal{Q}(\mathcal{\mathcal{G}})$, $\mathcal{A}=\mathcal{A}(\mathcal{\mathcal{G}})$, and $\mathcal{D}=\mathcal{D}(\mathcal{\mathcal{G}})$.

If $d_{1}=d_{2}$, then $d^*=1$, and by Lemma~\ref{Tensor2}, we have
\[
\mu(\mathcal{G})=\rho(\mathcal{Q})\leq\max_{1\leq i\leq n}r_{i}(\mathcal{Q})=\max_{1\leq i\leq n}2d_{i}=2d_{1}= d_{1}+d_{1}\left(\frac{1}{d^*}\right)^{k-1},
\]
and when $\mathcal{G}$ is connected, we have by Lemma~\ref{connec} that $\mathcal{Q}$ is weakly irreducible, and thus equality holds  if and only if $r_1(\mathcal{Q})=\dots =r_n(\mathcal{Q})$, i.e., $\mathcal{G}$ is a regular hypergraph.

Suppose in the following that $d_{1}> d_{2}$. Let $U=\mbox{diag}(t,1,\dots,1)$ be an $n\times n$ diagonal matrix, where $t>1$ is a variable to be determined later. Let $\mathcal{T}=U^{-(k-1)}\mathcal{Q}U$. By Lemma~\ref{Tensor1}, $\mathcal{Q}$ and $\mathcal{T}$ have the same real eigenvalues. By Lemma~\ref{irreducibility}, $\rho(\mathcal{Q})$ is an eigenvalue of $\mathcal{Q}$  and $\rho(\mathcal{T})$ is an eigenvalue of $\mathcal{T}$. Thus $\mu(\mathcal{G})=\rho(\mathcal{Q})=\rho(\mathcal{T})$.
%\begin{eqnarray*}
%(P^{-(k-1)}\mathcal{Q}P)_{i_{1}i_{2}\cdots i_{k}}&=&(P^{-(k-1)}AP)_{i_{1}i_{2}\cdots i_{k}}+(P^{-(k-1)}DP)_{i_{1}i_{2}\cdots i_{k}}\\
%&=&P_{i_{1}i_{1}}^{-(k-1)}A_{i_{1}i_{2}\cdots i_{k}}P_{i_{2}i_{2}}\cdots P_{i_{k}i_{k}}+D_{i_{1}i_{1}\cdots i_{1}}.
%\end{eqnarray*}
%Now take $P=diag(x,1,\cdots,1)$ with $x>1$. Then we have
We have
\begin{eqnarray*}
r_{1}(\mathcal{T})&=&\sum_{i_{2},\ldots,i_{k}\in[n]}T_{1i_{2}\dots i_{k}}\\
&=&\sum_{i_{2},\ldots,i_{k}\in[n]}U_{11}^{-(k-1)}A_{1i_{2}\ldots i_{k}}U_{i_{2}i_{2}}\cdots U_{i_{k}i_{k}}+D_{1\dots 1}\\
&=&\sum_{i_{2},\ldots,i_{k}\in[n]\setminus\{1\}}\frac{1}{t^{k-1}}A_{1i_{2}\ldots i_{k}}+d_{1}\\
&=&\frac{d_{1}}{t^{k-1}}+d_{1}.
\end{eqnarray*}
For $i\in [n]\setminus \{1\}$, let
$d_{1,i}=|\{e: 1,i\in e\in E(\mathcal{G})\}|$. Obviously, $d_{1,i}\le d_i$.
For $2\leq i\leq n$, we have
\begin{eqnarray*}
r_{i}(\mathcal{T})&=&\sum_{i_{2},\dots,i_{k}\in[n]}T_{ii_{2}\dots i_{k}}\\
&=&\sum_{i_{2},\dots,i_{k}\in[n]}U_{ii}^{-(k-1)}A_{ii_{2}\ldots i_{k}}U_{i_{2}i_{2}}\cdots U_{i_{k}i_{k}}+D_{i\dots i}\\
&=&\sum_{i_{2},\dots,i_{k}\in[n]\atop 1\in\{i_{2},\ldots,i_{k}\}}U_{ii}^{-(k-1)}A_{ii_{2}\ldots i_{k}}U_{i_{2}i_{2}}\cdots U_{i_{k}i_{k}}\\
&&+\sum_{i_{2},\dots,i_{k}\in[n]\atop 1\not\in\{i_{2},\ldots,i_{k}\}}U_{ii}^{-(k-1)}A_{ii_{2}\ldots i_{k}}U_{i_{2}i_{2}}\cdots U_{i_{k}i_{k}}+d_{i}\\
&=&\sum_{i_{2},\dots,i_{k}\in[n]\atop 1\in\{i_{2},\ldots,i_{k}\}}A_{ii_{2}\ldots i_{k}}t+\sum_{i_{2},\dots,i_{k}\in[n]\atop 1\not\in\{i_{2},\ldots,i_{k}\}}A_{ii_{2}\ldots i_{k}}+d_{i}\\
&=&td_{1,i}+d_{i}-d_{1,i}+d_i\\
&\leq& (t+1)d_{i}\\
&\leq& (t+1)d_{2}
\end{eqnarray*}
with equality if and only if $d_i=d_{1,i}$ and $d_i=d_2$.

Note that $h((\frac{d_{1}}{d_{2}})^{\frac{1}{k}})=(d_2-d_1)\left(\frac{d_1}{d_2}\right)^{1-\frac{1}{k}}<0$ and $h((\frac{d_{1}}{d_{2}}))=d_1\left(\left(\frac{d_1}{d_2}\right)^{k-2}-1\right)>0$. Thus  $h(t)=0$ does have a root  $d^*$ in $((\frac{d_{1}}{d_{2}})^{\frac{1}{k}},\frac{d_{1}}{d_{2}})$. Let  $t=d^*$. Then  $t>1$. We have
\[
r_{1}(\mathcal{T})=d_{1}+d_{1}\left(\frac{1}{d^*}\right)^{k-1},
\]
and for $2\leq i\leq n$,
\[
r_{i}(\mathcal{T})\leq d_{1}+d_{1}\left(\frac{1}{d^*}\right)^{k-1}.
\]
Thus  by Lemma~\ref{Tensor2},
\[
\mu(\mathcal{G})=\rho(\mathcal{T})\leq \max_{1\leq i\leq n}r_{i}(\mathcal{T})=d_{1}+d_{1}\left(\frac{1}{d^*}\right)^{k-1}.
\]
This proves (\ref{MZh}).

Now suppose that $\mathcal{G}$ is connected. By Lemma~\ref{connec}, $\mathcal{Q}$ is weakly irreducible, and so is $\mathcal{T}$.

If  equality holds in (\ref{MZh}), then by Lemma~\ref{Tensor2}, $r_1(\mathcal{T})=\cdots=r_n(\mathcal{T})$, and thus  from the above arguments,  we have $d_{1,i}=d_{i}$ for $i=2, \dots, n$ (implying that each edge of $\mathcal{G}$ contains vertex $1$), $d_{2}=\cdots=d_{n}$, and thus  $\mathcal{G}$ is a blow-up hypergraph of a regular $(k-1)$-uniform hypergraph on $n-1$ vertices of degree $d_2$.

Conversely,  if $\mathcal{G}=\mathcal{H}^1$, where $\mathcal{H}$ is a regular $(k-1)$-uniform hypergraph on $n-1$ vertices of degree $d_2$, then by the above arguments, we have   $r_{i}(\mathcal{T})=d_{1}+d_{1}(\frac{1}{d^*})^{k-1}$ for $1\leq i\leq n$, and thus by Lemma~\ref{Tensor2},
 $\mu(\mathcal{G})=\rho(\mathcal{Q})=\rho(\mathcal{T})=d_{1}+d_{1}\left(\frac{1}{d^*}\right)^{k-1}$.
\end{proof}

 Let $\mathcal{G}$ be a  $k$-uniform hypergraph on $n$ vertices with degree sequence $d_1\ge \cdots \ge d_n$. If $d_1>d_2$, then $d_{1}+d_{1}(\frac{1}{d^*})^{k-1}<d_{1}+d_{1}(\frac{d_2}{d_1})^{1-\frac{1}{k}}
=d_{1}+d_{1}^{\frac{1}{k}}d_{2}^{1-\frac{1}{k}}$.  By Theorem~\ref{Mo}, we have
\[
\mu(\mathcal{G})\le d_1+d_1^{\frac{1}{k}}d_2^{1-\frac{1}{k}},
\]
and if $\mathcal{G}$ is connected, then equality holds  if and only if $\mathcal{G}$ is a regular hypergraph, see  \cite{YZL}.

%Note that the upper bound in (\ref{MZh}) is less than or equal to the one in (\ref{Yuan}). In particular,
%if $d_1>d_2$, then $d_{1}+d_{1}(\frac{1}{\gamma})^{k-1}<d_{1}+d_{1}(\frac{d_2}{d_1})^{1-\frac{1}{k}}
%=d_{1}+d_{1}^{\frac{1}{r}}d_{2}^{1-\frac{1}{r}}$. Thus
%the upper bound in   Theorem~\ref{upp} is better than the one given in (\ref{Yuan}).

\begin{Theorem}\label{Tensor7}
Let $\mathcal{G}$ be a $k$-uniform hypergraph on $n$ vertices  without isolated vertices with average $2$-degrees $m_1\geq \dots \geq m_n$,  and degree sequence $d_1, \dots, d_n$.
Then
\begin{equation} \label{Final}
\mu(\mathcal{G})\leq \min_{1\leq j\leq n}\max\left\{m_1^{\frac{1}{k}}m_{j}^{1-\frac{1}{k}}+d_1, \theta_j\right\},
\end{equation}
where \[
\theta_j =\max\left\{ m_1^{\frac{1}{k}}m_{i}m_j^{-\frac{1}{k}}+d_i :2\leq i\leq n\right\}.
\]

\end{Theorem}
\begin{proof} Let $\mathcal{Q}=\mathcal{Q}(\mathcal{G})$.
Let $U$ be a diagonal matrix $\mbox{diag} (td_1, d_2,\dots, d_n)$, where $t\geq1$ is a variable to be determined later.
Let $\mathcal{T}=U^{-(k-1)}\mathcal{Q}U$.
Then $\mathcal{Q}$ and $\mathcal{T}$ are diagonal similar. By Lemma~\ref{Tensor1}, $\mu(\mathcal{G})=\rho(\mathcal{T})$.
Obviously,
\[T_{i_1\dots i_k}=U^{-(k-1)}_{i_1i_1}Q_{i_1\dots i_k}U_{i_2i_2}\cdots U_{i_ki_k}.\]
for $i_1,\dots, i_k\in[n]$.
Then
\begin{eqnarray*}
r_1(\mathcal{T})%&=& \sum_{i_2,\dots, i_k\in[n]}T_{ii_2\dots i_k}\\
&=&\sum_{i_2,\dots, i_k\in[n]}U^{-(k-1)}_{11}Q_{1i_2\dots i_k}U_{i_2i_2}\cdots U_{i_ki_k}\\
&=&\frac{\sum_{i_2,\dots, i_k\in[n]\setminus\{1\}}Q_{1i_2\dots i_k}d_{i_2}\cdots d_{i_k}}{(t d_1)^{k-1}}\\
&=&\frac{D_{1\dots 1}(td_1)^{k-1}}{(t d_1)^{k-1}}+\frac{\sum_{i_2,\dots, i_k\in[n]}A_{1i_2\dots i_k}d_{i_2}\cdots d_{i_k}}{(t d_1)^{k-1}}\\
&=&d_1+\frac{\sum_{\{1,i_2,\dots, i_k\}\in E_1}d_{i_2}\cdots d_{i_k}}{(t d_1)^{k-1}}\\
&=&d_1+\frac{m_1}{t^{k-1}}.
\end{eqnarray*}
For $i=2, \dots, n$, let \[m_{1,i}=m_i-\frac{\sum_{1\notin\{i,i_2,\dots,i_k\}\in E_i}d_{i_2}\dots d_{i_k}}{d_i^{k-1}},\]
and then
\begin{eqnarray*}
r_i(\mathcal{T})%&=& \sum_{i_2,\dots, i_k\in[n]}\mathcal{T}_{ii_2\dots i_k}\\
&=&\sum_{i_2,\dots, i_k\in[n]}U^{-(k-1)}_{ii}Q_{ii_2\dots i_k}U_{i_2i_2}\cdots U_{i_ki_k}\\
&=&\sum_{i_2,\dots, i_k\in[n]\atop 1\in\{i_2,\dots, i_k\}}d_i^{-(k-1)}Q_{ii_2\dots i_k}d_{i_2}\cdots d_{i_k}\\
&&+\sum_{i_2,\dots, i_k\in[n]\atop 1\notin \{i_2,\ldots , i_k\}}d_i^{-(k-1)}Q_{ii_2\dots i_k}d_{i_2}\cdots d_{i_k}\\
&=&\frac{D_{i\dots i}d_i^{k-1}}{ d_i^{k-1}}+\frac{\sum_{i_2,\dots, i_k\in[n]\atop 1\in\{i_2,\dots, i_k\}}A_{ii_2\dots i_k}d_{i_2}\cdots d_{i_k}}{d_i^{k-1}}\\
&&+\frac{\sum_{i_2,\dots, i_k\in[n]\atop 1\notin\{i_2,\dots, i_k\}}A_{ii_2\dots i_k}d_{i_2}d_{i_2}\cdots d_{i_k}}{d_i^{k-1}}\\
&=&d_i+\frac{\sum_{\{i,1,i_3,\dots, i_k\}\in E_i}(td_1)d_{i_3}\cdots d_{i_k}}{d_i^{k-1}}\\
&&+\frac{\sum_{1\notin\{i,i_2,\dots,i_k\}\in E_i}d_{i_2}\cdots d_{i_k}}{d_i^{k-1}}\\
&=&d_i+ tm_{1,i}+m_i-m_{1,i}\\
&\leq&d_i+ tm_i.
\end{eqnarray*}
For an arbitrary fixed $j$ with  $1\leq j\leq n$, let $t=\left(\frac{m_1}{m_j}\right)^\frac{1}{k}$. Obviously, $t\geq 1$.
Then \[r_1(\mathcal{T})=m_1^{\frac{1}{k}}m_{j}^{1-\frac{1}{k}}+d_1,\]
for $2\leq i\leq n$, \[r_i(\mathcal{T})\leq tm_i+d_i=m_1^{\frac{1}{k}}m_{i}m_j^{-\frac{1}{k}}+d_i.\]

Let $\theta_j =\max\left\{ m_1^{\frac{1}{k}}m_{i}m_j^{-\frac{1}{k}}+d_i :2\leq i\leq n\right\}$.
Thus for $1\leq i\leq n$, we have  \[r_i(\mathcal{T})\leq\max\left\{m_1^{\frac{1}{k}}m_{j}^{1-\frac{1}{k}}+d_1, \theta_j\right\}\]
Thus \[r_i(\mathcal{T})\leq \min_{1\leq j\leq n}\max\left\{m_1^{\frac{1}{k}}m_{j}^{1-\frac{1}{k}}+d_1, \theta_j\right\}.\]
Now the result follows from   Lemma~\ref{Tensor2}.
\end{proof}

%, \[\rho(\mathcal{G})=\rho(\mathcal{T})\leq \min_{1\leq j\leq n}\max\left\{m_1^{\frac{1}{k}}m_{j}^{1-\frac{1}{k}}+d_1, \theta_j\right\}.\]

If we  take $U=\mbox{diag} (d_1,\ldots,d_{n-1}, yd_n)$ %with $0<y\leq1$
%and $\mathcal{T}'=W^{-(k-1)}\mathcal{A}W$.
with $y=\left(\frac{m_n}{m_j}\right)^\frac{1}{k}$ for an arbitrary  fixed $j$ in the above proof, then we have
\[\mu(\mathcal{G})\ge \max_{1\leq j\leq n}\min\left\{m_n^{\frac{1}{k}}m_{j}^{1-\frac{1}{k}}+d_n, \gamma_j\right\},
\]
where $\gamma_j =\min\left\{ m_n^{\frac{1}{k}}m_{i}m_j^{-\frac{1}{k}}+d_i :2\leq i\leq n\right\}$ for $1\leq j\leq n$.

Consider $4$-uniform hypergraph $\mathcal{G}_1$ with vertex set $[25]$ and edge set $E(\mathcal{G}_1)=\{e_1,\ldots,e_{14}\}$, where \begin{eqnarray*}
\begin{matrix}
e_1=\{1,2,3,4\}, & e_2=\{5,6,7,8\}, & e_3=\{9,10,11,12\},\\
e_4=\{13,14,15,16\}, &
e_5=\{17,18,19,20\}, & e_6=\{21,22,23,24\}, \\
e_7=\{1,2,3,25\}, & e_8=\{4,5,6,25\}, &
e_9=\{7,8,9,25\},\\
 e_{10}=\{10,11,12,25\}, & e_{11}=\{13,14,15,25\}, & e_{12}=\{16,17,18,25\}\\
e_{13}=\{19,20,21,25\}, & e_{14}=\{22,23,24,25\}.
\end{matrix}
\end{eqnarray*}
In notation of Theorem~\ref{Tensor7}, we have
\[
d_1=\dots=d_{24}=2, d_{25}=8,
\]
and
\[
m_1=\dots=m_{24}=5, m_{25}= 0.125,
\]
implying that $\theta_1=\dots=\theta_{24}=14.5743$, $\theta_{25}\approx 8.31436$,  and $m_1^{\frac{1}{4}}m_j^{\frac{3}{4}}+d_1=7$ for $1\leq j\leq 24$ and $m_1^{\frac{1}{4}}m_{25}^{\frac{3}{4}}+d_1\approx 2.31436$. Thus $\mu(\mathcal{G}_1)\leq 8.125$.
Note that $8 =d_1>d_2=\dots=d_{25}=2$ in notation of Theorem \ref{Mo} and  that $h(d^*)=d_{1}+d_{1}\left(\frac{1}{d^*}\right)^{k-1}$ is a decreasing function for $d^*\in((\frac{d_{1}}{d_{2}})^{\frac{1}{k}},\frac{d_{1}}{d_{2}})$. Then
\[
d_{1}+d_{1}\left(\frac{1}{d^{\ast}}\right)^{k-1}
> d_{1}+d_{1}\left(\frac{d_2}{d_1}\right)^{k-1}
= d_{1}+\frac{d_2^{k-1}}{d_1^{k-2}}
= 8+\frac{2^3}{8^2}
=8.125.
\]
For $\mathcal{G}_1$, the upper bound in (\ref{Final}) is smaller than the one in (\ref{MZh}). Obviously,
the blow-up hypergraph of a regular $(k-1)$-uniform hypergraph on $n-1$ vertices, the upper bound in (\ref{MZh}) is smaller than the one in (\ref{Final}).

For a $k$-uniform hypergraph $\mathcal{G}$,
let $d_1\geq \dots\geq d_n$ be the degree sequence of $\mathcal{G}$ and $m_1,\dots, m_n$ be the average $2$-degrees of $\mathcal{G}$. In~\cite{YZL}, the following upper bounds for $\mu(\mathcal{G})$ are given.
%\begin{eqnarray}\label{TT1}
%\mu(\mathcal{G})\leq d_1+d_1^{\frac{1}{k}}d_2^{1-\frac{1}{k}},
%\end{eqnarray}
\begin{eqnarray}\label{TT2}
\mu(\mathcal{G})\leq \max_{e\in E(\mathcal{G})}\max_{\{i,j\}\in e}(d_i+d_j),
\end{eqnarray}
\begin{eqnarray}\label{TT3}
\mu(\mathcal{G})\leq \max_{e\in E(\mathcal{G})}\max_{\{i,j\}\in e}\frac{d_i+d_j+\sqrt{(d_i-d_j)^2+4m_im_j}}{2}.
\end{eqnarray}

%Consider $3$-uniform hypergraph $\mathcal{G}_1$ with vertex set $[6]$ and edge set $E(\mathcal{G})=\{e_1,\ldots,e_5\}$, where $e_1=\{1,2,6\}, e_2=\{2,3,6\},e_3=\{3,4,6\},e_4=\{4,5,6\},e_5=\{1,5,6\}$.
%In notation of Theorem~\ref{Tensor7}, $d_1=\dots=d_5=2$, $d_6=5$, $m_1=\dots =m_5=5$, and $m_6=\frac{4}{5}$,
%implying that $\theta_1=\dots=\theta_5=7$, $\theta_6\approx11.21$, and $m_1^{\frac{1}{3}}m_j^{\frac{2}{3}}+d_1= 7$ when $1\leq j\leq 5$ and $m_1^{\frac{1}{3}}m_j^{\frac{2}{3}}+d_1\approx6.4736$ when $j=6$, and thus $\mu(\mathcal{G}_1)\leq 7$.
%While in (\ref{TT1}),(\ref{TT2}) and (\ref{TT3}),  we have \[d_1+d_1^{\frac{1}{k}}d_2^{1-\frac{1}{k}}\approx7.7144 ,\]
%\[\max_{e\in E(\mathcal{G})}\max_{\{i,j\}\in e}(d_i+d_j)=7 ,\]
%and \[ \max_{e\in E(\mathcal{G})}\max_{\{i,j\}\in e}\frac{d_i+d_j+\sqrt{(d_i-d_j)^2+4m_im_j}}{2}=7.\]
%The upper bound in Theorem~\ref{Tensor7} is smaller than that of (\ref{TT1}).

Consider $3$-uniform hypergraph $\mathcal{G}_2$ with vertex set $[9]$ and edge set $E(\mathcal{G}_2)=\{e_1,\ldots,e_4\}$, where
\[
e_1=\{1,2,9\}, e_2=\{3,4,8\},e_3=\{5,6,7\},e_4=\{7,8,9\}.
\]
In notation of Theorem~\ref{Tensor7}, we have
\[
d_1=\dots=d_6=1, d_7=d_8=d_9=2,
\]
and
\[
m_1=\dots =m_6=2, m_7=m_8=m_9=\frac{5}{4},
\]
implying that $\theta_1=\dots=\theta_6=3.25$, $\theta_7=\theta_8=\theta_9\approx3.462$, and $m_1^{\frac{1}{3}}m_j^{\frac{2}{3}}+d_1=3$ when $1\leq j\leq 6$ and $m_1^{\frac{1}{3}}m_j^{\frac{2}{3}}+d_1\approx2.462$ when $7\leq j\leq 9$, and thus $\mu(\mathcal{G}_2)\leq 3.25$.
By direct calculation, the bounds in (\ref{TT2}) and (\ref{TT3}) are  $4$ and $3.25$, respectively.
%
%
%While in (\ref{TT1}),(\ref{TT2}) and (\ref{TT3}),  we have \[d_1+d_1^{\frac{1}{k}}d_2^{1-\frac{1}{k}}=4 ,\]
%\[\max_{e\in E(\mathcal{G})}\max_{\{i,j\}\in e}(d_i+d_j)=4 ,\]
%and \[ \max_{e\in E(\mathcal{G})}\max_{\{i,j\}\in e}\frac{d_i+d_j+\sqrt{(d_i-d_j)^2+4m_im_j}}{2}=3.25.\]
For $\mathcal{G}_2$, the upper bound in (\ref{Final}) is smaller than the upper bound in (\ref{TT2}).

Consider $4$-uniform hypergraph $\mathcal{G}_3$ with vertex set $[7]$ and edge set $E(\mathcal{G}_3)=\{e_1,\ldots,e_{8}\}$, where \begin{eqnarray*}
\begin{matrix}
e_1=\{1,2,3,4\}, & e_2=\{1,5,6,7\}, & e_3=\{2,3,4,5\}, \\
e_4=\{3,4,5,6\}, & e_5=\{4,5,6,7\}, & e_6=\{5,6,7,2\}, \\
e_7=\{6,7,2,3\}, & e_8=\{7,2,3,4\}.
\end{matrix}
\end{eqnarray*}
In notation of Theorem~\ref{Tensor7}, we have
\[
d_1=2, d_2=\dots=d_7=5,
\]
and
\[
m_1=31.25, m_2=\dots=m_7=4.4,
\]
implying that $\theta_1=9.4$, and $\theta_2=\dots=\theta_7\approx 12.18294$, and $m_1^{\frac{1}{4}}m_j^{\frac{3}{4}}+d_1=33.25$ for  $j=1$ and $m_1^{\frac{1}{4}}m_j^{\frac{3}{4}}+d_1\approx9.18294$ for $2\leq j\leq 7$. Thus $\mu(\mathcal{G}_3)\leq 12.18294$. It is easily seen that  the bounds in (\ref{TT2}) and (\ref{TT3}) are $10$ and $15.32159$, respectively.
%While in (\ref{TT1}),(\ref{TT2}) and (\ref{TT3}),  we have
%\[d_1+d_1^{\frac{1}{k}}d_2^{1-\frac{1}{k}}=10 ,\]
%\[\max_{e\in E(\mathcal{G})}\max_{\{i,j\}\in e}(d_i+d_j)=10 ,\]
%and
%\[ \max_{e\in E(\mathcal{G})}\max_{\{i,j\}\in e}\frac{d_i+d_j+\sqrt{(d_i-d_j)^2+4m_im_j}}{2}\approx 15.32159.\]
For $\mathcal{G}_3$, the upper bound in Theorem~\ref{Tensor7} is smaller than the one in  (\ref{TT3}) but larger than the one in  (\ref{TT2}). %, and larger than that of both (\ref{TT1}) and (\ref{TT2}) simultaneously.

\vspace{5mm}

\noindent {\bf Acknowledgement.}  We thank the referees for their valuable and constructive comments and suggestions.
%This work was supported by the Specialized Research Fund for the Doctoral Program of Higher Education of China %(No.~20124407110002).


\begin{thebibliography}{99}


\bibitem{Be} C. Berge, Hypergraphs: Combinatorics of Finite Sets,  North-Holland, Amsterdam, 1989.

\bibitem{CD} J. Cooper, A. Dutle, Spectra of uniform hypergraphs, Linear Algebra Appl.  436  (2012) 3268--3292.


\bibitem{FGH} S. Friedland, S. Gaubert, L. Han, Perron-Frobenius theorems for nonnegative multilinear forms and extension, Linear Algebra Appl. 438 (2013)
738--749.


\bibitem{FYZ} M. Khan, Y.Z. Fan, On the spectral radius of a class of non-odd-bipartite even uniform hypergraphs, Linear Algebra Appl. 480 (2015) 93--106.


\bibitem{LCL} C. Li, Z. Chen, Y. Li, A new eigenvalue inclusion set for tensors and its applications, Linear Algebra Appl. 481 (2015) 36--53.


\bibitem{PT} K. Pearson, T. Zhang, On spectral hypergraph theory of the adjacency tensor, Graphs Combin.  30  (2014) 1233--1248.


\bibitem{Q1} L. Qi,  Eigenvalues of a real supersymmetric tensor,  J. Symbol Comput. 40 (2005) 1302--1324.

\bibitem{Qi} L. Qi, $H^+$-eigenvalues of Laplacian and signless Laplacian tensors, Commun. Math. Sci. 12 (2014) 1045--1064.

\bibitem{QSW} L. Qi, J.-Y. Shao, Q. Wang, Regular uniform hypergraphs, $s$-cycles, $s$-paths and their largest Laplacian eigenvalues,
Linear Algebra Appl.  443  (2014) 215--227.

\bibitem{Sh} J.-Y. Shao, A general product of tensors with applications,
Linear Algebra Appl. 439 (2012) 2350--2366.

\bibitem{YY} Y. Yang, Q. Yang, Further results for Perron-Frobenius theorem for nonegative tensors, SIAM J. Matrix Anal. Appl. 31 (2010) 2517--2530.

\bibitem{YZL} X. Yuan, M. Zhang, M. Lu, Some upper bounds on the eigenvalues of uniform hypergraphs,
Linear Algebra Appl. 484 (2015) 540--549.



\end{thebibliography}
\end{document}